\documentclass[12pt]{amsart}
\usepackage{amsmath}
\usepackage{amssymb}
\usepackage{amsfonts}

\newtheorem{theorem}{Theorem}
\theoremstyle{plain}

\newtheorem{corollary}{Corollary}

\newtheorem{lemma}{Lemma}

\newtheorem{remark}{Remark}

\numberwithin{equation}{section}
\input{tcilatex}

\begin{document}
\title[]{The quotient shapes of normed spaces and application}
\author{Nikica Ugle\v{s}i\'{c} }
\address{ Veli R\aa t, Dugi Otok, Hrvatska}
\email{nuglesic@unizd.hr}
\date{February 26, 2018}
\subjclass[2000]{[2010] Primary 54E99; Secondary 55P55 }
\keywords{normed (Banach, Hilbert) vectorial space, quotient normed space,
algebraic dimension, expansion, (infinite) cardinal, (general) continuum
hypothesis, quotient shape, continuous linear extension. }
\dedicatory{}
\thanks{This paper is in final form and no version of it will be submitted
for publication elsewhere.}

\begin{abstract}
The quotient shape types of normed vectorial spaces (over the same field)
with respect to Banach spaces reduce to those of Banach spaces. The finite
quotient shape type of normed spaces is an invariant of the (algebraic)
dimension, but not conversely. The converse holds for separable normed
spaces as well as for the bidual-like spaces (isomorphic to their second
dual spaces). As a consequence, the Hilbert space $l_{2}$, or even its
(countably dimensional, unitary) direct sum subspace may represent the
unique quotient shape type of all $2^{\aleph _{0}}$-dimensional normed
spaces. An application yields two extension type theorems.
\end{abstract}

\maketitle

\section{Introduction}

The shape theory began as a generalization of the homotopy theory such that
the locally bad spaces can be also considered and classified in a very
suitable \textquotedblleft homotopical\textquotedblright\ way. The first
step (for compacta in the Hilbert cube) had made K. Borsuk, [1]. The theory
was rapidly developed and generalized by many authors. The main references
are [2], [3], [5] and, especially, [10]. Although, in general, founded
purely categorically, a shape theory is mostly well known only as \emph{the}
(standard) shape theory of topological spaces with respect to spaces having
the homotopy types of polyhedra. The generalizations founded in [7] and [16]
are, primarily, also on that line.

The quotient shape theory was introduced a few years ago by the author,
[12]. It is, of course, a kind of the general (abstract) shape theory, [10],
I. 2. However, it is possible and non-trivial, and can be straightforwardly
developed for every concrete category $\mathcal{C}$ and for every infinite
cardinal $\kappa \geq \aleph _{0}$. Concerning a shape of objects, in
general, one has to decide which ones are \textquotedblleft
nice\textquotedblright\ absolutely and/or relatively (with respect to the
chosen ones). In this approach, the main principle reads as follows: \emph{%
An object is \textquotedblleft nice\textquotedblright\ if it is isomorphic
to a quotient object belonging to a special full subcategory and if it (its
\textquotedblleft basis\textquotedblright ) has cardinality less than (less
than or equal to) a given infinite cardinal. }It leads to the basic idea: to
approximate a $\mathcal{C}$-object $X$ by a suitable inverse system
consisting of its quotient objects $X_{\lambda }$ (and the quotient
morphisms) which have cardinalities, or dimensions - in the case of
vectorial spaces, less than (less than or equal to) $\kappa $. Such an
approximation exists in the form of any $\kappa ^{-}$-expansion ($\kappa $%
-expansion) of $X$,

$\boldsymbol{p}_{\kappa ^{-}}=(p_{\lambda }):X\rightarrow \boldsymbol{X}%
_{\kappa ^{-}}=(X_{\lambda },p_{\lambda \lambda ^{\prime }},\Lambda _{\kappa
^{-}})$

($\boldsymbol{p}_{\kappa }=(p_{\lambda }):X\rightarrow \boldsymbol{X}%
_{\kappa }=(X_{\lambda },p_{\lambda \lambda ^{\prime }},\Lambda _{\kappa })$%
),

\noindent where $\boldsymbol{X}_{\kappa ^{-}}$ ($\boldsymbol{X}_{\kappa }$)
belongs to the subcategory $pro$-$\mathcal{D}_{\kappa ^{-}}$ ($pro$-$%
\mathcal{D}_{\kappa }$) of $pro$-$\mathcal{D}$, and $\mathcal{D}_{\kappa
^{-}}$ ($\mathcal{D}_{\kappa }$) is the subcategory of $\mathcal{D}$
determined by all the objects having cardinalities, or dimensions - for
vectorial spaces, less than (less than or equal to) $\kappa $, while $%
\mathcal{D}$ is a full subcategory of $\mathcal{C}$. Clearly, if $X\in Ob(%
\mathcal{D})$ and the cardinality $\left\vert X\right\vert <\kappa $ ($%
\left\vert X\right\vert \leq \kappa $), then the rudimentary pro-morphism $%
\left\lfloor 1_{X}\right\rfloor :X\rightarrow \left\lfloor X\right\rfloor $
is a $\kappa ^{-}$-expansion ($\kappa $-expansion) of $X$. The corresponding
shape category $Sh_{\mathcal{D}_{\kappa ^{-}}}(\mathcal{C})$ ($Sh_{\mathcal{D%
}_{\kappa }}(\mathcal{C})$) and shape functor $S_{\kappa ^{-}}:\mathcal{C}%
\rightarrow Sh_{\mathcal{D}_{\kappa ^{-}}}(\mathcal{C})$ ($S_{\kappa }:%
\mathcal{C}\rightarrow Sh_{\mathcal{D}_{\kappa }}(\mathcal{C})$) exist by
the general (abstract) shape theory, and they have all the appropriate
general properties. Moreover, there exist the relating functors $S_{\kappa
^{-}\kappa }:Sh_{\mathcal{D}_{\kappa }}(\mathcal{C})\rightarrow Sh_{\mathcal{%
D}_{\kappa ^{-}}}(\mathcal{C})$ and $S_{\kappa \kappa ^{\prime }}:Sh_{%
\mathcal{D}_{\kappa ^{\prime }}}(\mathcal{C})\rightarrow Sh_{\mathcal{D}%
_{\kappa }}(\mathcal{C})$, $\kappa \leq \kappa ^{\prime }$, such that $%
S_{\kappa ^{-}\kappa }S_{\kappa }=S_{\kappa ^{-}}$ and $S_{\kappa \kappa
^{\prime }}S_{\kappa ^{\prime }}=S_{\kappa }$. Even in simplest case of $%
\mathcal{D}=\mathcal{C}$, the quotient shape classifications are very often
non-trivial and very interesting. In such a case we simplify the notation $%
Sh_{\mathcal{D}_{\kappa ^{-}}}(\mathcal{C})$ ($Sh_{\mathcal{D}_{\kappa }}(%
\mathcal{C})$) to $Sh_{\kappa ^{-}}(\mathcal{C})$ ($Sh_{\kappa }(\mathcal{C}%
) $) or to $Sh_{\kappa ^{-}}$ ($Sh_{\kappa }$) when $\mathcal{C}$ is fixed.

In [12], several well known concrete categories were considered and many
examples are given which show that the quotient shape theory yields
classifications strictly coarser than those by isomorphisms. In [13] and
[14] were considered the quotient shapes of (purely algebraic, topological
and normed) vectorial spaces and topological spaces, respectively. In the
recent paper [15], we have continued the studying of quotient shapes of
normed vectorial spaces of [13], Section 4.1, primarily and \emph{separately}
focused to the well known $l_{p}$ and $L_{p}$ spaces and to the Sobolev
spaces $W_{p}^{(k)}(\Omega _{n})$ (of all real functions on $\Omega _{n}$
having their supports in a domain $\Omega _{n}$ and all partial derivatives
up to order $k$ continuous). The main global result of [15] is that the 
\emph{finite} quotient shape type of a normed spaces (over the field $F\in \{%
\mathbb{R},\mathbb{C}\}$) reduces to that of its completion (Banach) spaces,
and consequently, that the quotient shape theory of $(NVect_{F},(NVect_{F})_{%
\text{\b{0}}})$ reduces to that of $(BVect_{F},(BVect_{F})_{\text{\b{0}}})$

In this work we have clarified the relationship between the quotient shape
theories of normed and that of Banach spaces (Theorem 1). Further, we have
proven that the finite quotient shape type of normed spaces is an invariant
of the (algebraic) dimension, but not conversely. The counterexamples exist
in the dimensional par $\{\aleph _{0},2^{\aleph _{0}}\}$. In the case of
separable Banach spaces, the classifications by dimension and by the finite
quotient shape (as well as by the countable quotient shape) coincide
(Theorem 2). Consequently, all the infinite-dimensional separable normed
spaces over the same field belong to a unique quotient shape type with
respect to Banach spaces. Its representative may be, for instance, the
Hilbert space $l_{2}$, or its (countably infinite dimensional, unitary)
direct sum subspace $F_{0}^{\mathbb{N}}(2)\equiv (F_{0}^{\mathbb{N}%
},\left\Vert \cdot \right\Vert _{2})$ (Corollary 1).

An application yields two extension type theorems for the category $s%
\mathcal{B}\cup b\mathcal{B}$ ($s\mathcal{B}$ - separable Banach; $b\mathcal{%
B}$ - bidual-like Banach spaces, i.e., $X\cong X^{\ast \ast }$), provided a
subspace has the top-dimensional closure and a lower codimension, into lower
dimensional Banach spaces (Theorems 3 and its \textquotedblleft operable
realization\textquotedblright\ - Theorem 4).

\section{Preliminaries}

We shall frequently use and apply in the sequel several general or special
well known facts without referring to any source. So we remind a reader that

\noindent - our general shape theory technique is that of [10];

\noindent - the needed set theoretic (especially, concerning cardinals) and
topological facts can be found in [4];

\noindent - the facts concerning functional analysis are taken from [8], [9]
or [11];

\noindent - our category theory language follows that of [6].

For the sake of completeness, let us briefly repeat the construction of a
quotient shape category and a quotient shape functor, [12]. Given a category
pair $(\mathcal{C},\mathcal{D}$), where $\mathcal{D}\subseteq \mathcal{C}$
is full, and a cardinal $\kappa $, let $\mathcal{D}_{\kappa ^{-}}$ ($%
\mathcal{D}_{\kappa }$) denote the full subcategory of $\mathcal{D}$
determined by all the objects having cardinalities or, in some special
cases, the cardinalities of \textquotedblleft bases\textquotedblright\ less
than (less or equal to) $\kappa $. By following the main principle, let $(%
\mathcal{C},\mathcal{D}_{\kappa ^{-}})$ ($(\mathcal{C},\mathcal{D}_{\kappa
}) $) be such a pair of \emph{concrete} categories. If

\noindent (a) every $\mathcal{C}$-object $(X,\sigma )$ admits a directed set 
$R(X,\sigma ,\kappa ^{-})\equiv \Lambda _{\kappa ^{-}}$ ($R(X,\sigma ,\kappa
)\equiv \Lambda _{\kappa }$) of equivalence relations $\lambda $ on $X$ such
that each quotient object $(X/\lambda ,\sigma _{\lambda })$ has to belong to 
$\mathcal{D}_{\kappa ^{-}}$ ($\mathcal{D}_{\kappa }$), while each quotient
morphism $p_{\lambda }:(X,\sigma )\rightarrow (X/\lambda ,\sigma _{\lambda
}) $ has to belong to $\mathcal{C}$;

\noindent (b) the induced morphisms between quotient objects belong to $%
\mathcal{D}_{\kappa ^{-}}$ ($\mathcal{D}_{\kappa }$);

\noindent (c) every morphism $f:(X,\sigma )\rightarrow (Y,\tau )$ of $%
\mathcal{C}$, having the codomain in $\mathcal{D}_{\kappa ^{-}}$ ($\mathcal{D%
}_{\kappa }$), factorizes uniquely through a quotient morphism $p_{\lambda
}:(X,\sigma )\rightarrow (X/\lambda ,\sigma _{\lambda })$, $f=gp_{\lambda }$%
, with $g$ belonging to $\mathcal{D}_{\kappa ^{-}}$ (\thinspace $\mathcal{D}%
_{\kappa }$),

\noindent then $\mathcal{D}_{\kappa ^{-}}$ ($\mathcal{D}_{\kappa }$) is a
pro-reflective subcategory of $\mathcal{C}$. Consequently, there exists a
(non-trivial) \emph{(quotient)\ shape\ category }$Sh_{(\mathcal{C},\mathcal{D%
}_{\kappa ^{-}})}\equiv Sh_{\mathcal{D}_{\kappa ^{-}}}(\mathcal{C})$ ($Sh_{(%
\mathcal{C},\mathcal{D}_{\kappa })}\equiv Sh_{\mathcal{D}_{\kappa }}(%
\mathcal{C})$) obtained by the general construction.

Therefore, a $\kappa ^{-}$-shape morphism $F_{\kappa ^{-}}:(X,\sigma
)\rightarrow (Y,\tau )$ is represented by a diagram (in $pro$-$\mathcal{C}$)

$%
\begin{array}{ccc}
(\boldsymbol{X},\boldsymbol{\sigma })_{\kappa ^{-}} & \overset{p_{\kappa
^{-}}}{\leftarrow } & (X,\sigma ) \\ 
\boldsymbol{f}_{\kappa ^{-}}\downarrow &  &  \\ 
(\boldsymbol{Y},\boldsymbol{\tau })_{\kappa ^{-}} & \overset{q_{\kappa ^{-}}}%
{\leftarrow } & (Y,\tau )%
\end{array}%
$

\noindent (with $\boldsymbol{p}_{\kappa ^{-}}$ and $\boldsymbol{q}_{\kappa
^{-}}$ - a pair of appropriate expansions), and similarly for a $\kappa $%
-shape morphism $F_{\kappa }:(X,\sigma )\rightarrow (Y,\tau )$. Since all $%
\mathcal{D}_{\kappa ^{-}}$-expansions ($\mathcal{D}_{\kappa }$-expansions)
of a $\mathcal{C}$-object are mutually isomorphic objects of $pro$-$\mathcal{%
D}_{\kappa ^{-}}$ ($pro$-$\mathcal{D}_{\kappa }$), the composition and
identities follow straightforwardly. Observe that every quotient morphism $%
p_{\lambda }$ is an effective epimorphism. (If $U$ is the forgetful functor,
then $U(p_{\lambda })$ is a surjection), and thus condition (E2) for an
expansion follows trivially.

The corresponding \emph{\textquotedblleft quotient\ shape\textquotedblright\
functors} $S_{\kappa ^{-}}:\mathcal{C}\rightarrow Sh_{\mathcal{D}_{\kappa
^{-}}}(\mathcal{C})$ and $S_{\kappa }:\mathcal{C}\rightarrow Sh_{\mathcal{D}%
_{\kappa }}(\mathcal{C})$ are defined in the same general manner. That means,

$S_{\kappa ^{-}}(X,\sigma )=S_{\kappa }(X,\sigma )=(X,\sigma )$;

if $f:(X,\sigma )\rightarrow (Y,\tau )$ is a $\mathcal{C}$-morphism, then,
for every $\mu \in M_{\kappa ^{-}}$, the composite $g_{\mu }f:(Y,\tau
)\rightarrow (Y_{\mu },\tau _{\mu })$ factorizes (uniquely) through a $%
p_{\lambda (\mu )}:(X,\sigma )\rightarrow (X_{\lambda (\mu )},\sigma
_{\lambda (\mu )}),$ and thus, the correspondence $\mu \mapsto \lambda (\mu
) $ yields a function $\phi :M_{\kappa ^{-}}\rightarrow \Lambda _{\kappa
^{-}}$ and a family of $\mathcal{D}_{\kappa ^{-}}$-morphisms $f_{\mu
}:(X_{\phi (\mu )},\sigma _{\phi (\mu )})\rightarrow (Y_{\mu },\tau _{\mu })$
such that $q_{\mu }f=f_{\mu }p_{\phi (\mu )}$;

\noindent one easily shows that $(\phi ,f_{\mu }):(\boldsymbol{X},%
\boldsymbol{\sigma })_{\kappa ^{-}}\rightarrow (\boldsymbol{Y},\boldsymbol{%
\tau })_{\kappa ^{-}}$ is a morphism of $inv$-$\mathcal{D}_{\kappa ^{-}}$,
so the equivalence class $\boldsymbol{f}_{\kappa ^{-}}=[(\phi ,f_{\mu })]:(%
\boldsymbol{X},\boldsymbol{\sigma })_{\kappa ^{-}}\rightarrow (\boldsymbol{Y}%
,\boldsymbol{\tau })_{\kappa ^{-}}$ is a morphism of $pro$-$\mathcal{D}%
_{\kappa ^{-}}$;

\noindent then we put $S_{\kappa ^{-}}(f)=\left\langle \boldsymbol{f}%
_{\kappa ^{-}}\right\rangle \equiv F_{\kappa ^{-}}:(X,\sigma )\rightarrow
(Y,\tau )$ in $Sh_{\mathcal{D}_{\kappa ^{-}}}(\mathcal{C})$.

\noindent The identities and composition are obviously preserved. In the
same way one defines the functor $S_{\kappa }$.

Furthermore, since $(\boldsymbol{X,\sigma })_{\kappa ^{-}}$ is a subsystem
of $(\boldsymbol{X,\sigma })_{\kappa }$ (more precisely, $(\boldsymbol{%
X,\sigma })_{\kappa }$ is a subobject of $(\boldsymbol{X,\sigma })_{\kappa
^{-}}$ in $pro$-$\mathcal{D}$), one easily shows that there exists a functor 
$S_{\kappa ^{-}\kappa }:Sh_{\mathcal{D}_{\kappa }}(\mathcal{C})\rightarrow
Sh_{\mathcal{D}_{\kappa ^{-}}}(\mathcal{C})$ such that $S_{\kappa ^{-}\kappa
}S_{\kappa }=S_{\kappa ^{-}}$, i.e., the diagram

$%
\begin{array}{ccccc}
&  & \mathcal{C} &  &  \\ 
& \swarrow S_{\kappa ^{-}} &  & S_{\kappa }\searrow &  \\ 
Sh_{\mathcal{D}_{\kappa ^{-}}}(\mathcal{C}) &  & \underleftarrow{S_{\kappa
^{-}\kappa }} &  & Sh_{\mathcal{D}_{\kappa }}(\mathcal{C})%
\end{array}%
$

\noindent commutes. Moreover, an analogous functor $S_{\kappa \kappa
^{\prime }}:Sh_{\mathcal{D}_{\kappa ^{\prime }}}(\mathcal{C})\rightarrow Sh_{%
\mathcal{D}_{\kappa }}(\mathcal{C})$, satisfying $S_{\kappa \kappa ^{\prime
}}S_{\kappa ^{\prime }}=S_{\kappa }$, exists for every pair of infinite
cardinals $\kappa \leq \kappa ^{\prime .}$

Generally, in the case of $\kappa =\aleph _{0}$, the $\kappa ^{-}$-shape is
said to be the \emph{finite (quotient) shape}, because all the objects in
the expansions are of finite (bases) cardinalities, and the category is
denoted by $Sh_{\mathcal{D}_{\underline{0}}}(\mathcal{C})$ or by $Sh_{\text{%
\b{0}}}(\mathcal{C})\equiv Sh_{\underline{0}}$ only, whenever $\mathcal{D}=%
\mathcal{C}$.

Let us finally notice that, though $\mathcal{D}\nsubseteq \mathcal{C}%
_{\kappa ^{-}})$ ($(\mathcal{D\nsubseteq C}_{\kappa })$), the quotient shape
category $Sh_{\mathcal{C}_{\kappa ^{-}}}(\mathcal{D})$ ($Sh_{\mathcal{C}%
_{\kappa }}(\mathcal{D})$) exists as a full subcategory of $Sh_{\mathcal{C}%
_{\kappa ^{-}}}(\mathcal{C})$ ($Sh_{\mathcal{C}_{\kappa }}(\mathcal{C})$),
and, if $\mathcal{D}$ is closed with respect to quotients, then $Sh_{%
\mathcal{C}_{\kappa ^{-}}}(\mathcal{D})=Sh_{\mathcal{D}_{\kappa ^{-}}}(%
\mathcal{D})$ ($Sh_{\mathcal{C}_{\kappa }}(\mathcal{D})=Sh_{\mathcal{D}%
_{\kappa }}(\mathcal{D})$).

\section{The quotient shapes of normed and Banach spaces}

Let $\mathcal{N}$ denote the category of all normed (vectorial) spaces over
the field $F\in \{\mathbb{R},\mathbb{C}\}$ together with all corresponding
continuous linear functions. Let $\mathcal{H}\subseteq \mathcal{B}\subseteq 
\mathcal{N}$ denote its full subcategories of all Hilbert and all Banach
spaces over the same $F$, respectively, and let $s\mathcal{N}$ ($b\mathcal{N}%
\subseteq \mathcal{B}$) denote full subcategory of all separable normed
(bidual-like) spaces over the same $F$. Further, given an infinite cardinal $%
\kappa \geq \aleph _{0}$, let $\mathcal{H}_{\kappa ^{-}}\subseteq \mathcal{B}%
_{\kappa ^{-}}\subseteq \mathcal{N}_{\kappa ^{-}}$ ($\mathcal{H}_{\kappa
}\subseteq \mathcal{B}_{\kappa }\subseteq \mathcal{N}_{\kappa }$) denote the
corresponding full subcategories determined by all the objects having \emph{%
algebraic} dimensions less than (less or equal to) $\kappa $.

The next lemma and theorem slightly reinforce Theorem 3 of [15].

\begin{lemma}
\label{L1}Let $Z$ be a dense subspace of a normed space $X$ and let $\dim
X\geq \kappa \geq \aleph _{0}$. Then, for every $\mathcal{B}_{\kappa ^{-}}$%
-expansion

$\boldsymbol{p}_{\kappa ^{-}}=(p_{\lambda }):X\rightarrow \boldsymbol{X}%
_{\kappa ^{-}}=(X_{\lambda },p_{\lambda \lambda ^{\prime }},\Lambda _{\kappa
^{-}})$

\noindent of $X$, the pro-morphism

$\boldsymbol{q}_{\kappa ^{-}}=(q_{\lambda }=p_{\lambda }j:Z\rightarrow 
\boldsymbol{X}_{\kappa ^{-}}$,

\noindent ($j:Z\hookrightarrow X$ is the inclusion) is a $\mathcal{B}%
_{\kappa ^{-}}$-expansion of $Z$. Especially, every continuous linear
function $f:Z\rightarrow Y$, where $Y$ is a Banach space over the same field
having $\dim Y<\kappa $, and its extension $\bar{f}:X\rightarrow Y$
factorize uniquely (linearly and continuously) trough a Banach space $%
X_{\lambda }$, $\dim X_{\lambda }<\kappa $, $\lambda \in \Lambda _{\kappa
^{-}}$, and if $\dim Y<2^{\aleph _{0}}\leq \kappa $, then $\dim X_{\lambda
}<\aleph _{0}$. The quite analogous statements hold in the $\kappa $-case, $%
\aleph _{0}\leq \kappa <\dim X$.
\end{lemma}

\begin{proof}
One readily sees that

$\boldsymbol{q}_{\kappa ^{-}}=(q_{\lambda }=p_{\lambda }j):Z\rightarrow 
\boldsymbol{X}_{\kappa ^{-}}$,

\noindent is a $\mathcal{B}_{\kappa ^{-}}$-expansion of $Z$. We need to
verify the mentioned factorization property. Let $Y$ be a Banach space over
the same field, $\dim Y<\kappa $, let $f:X\rightarrow Y$ be a continuous
linear function and let $\bar{f}:Cl(X)\rightarrow Y$ be the continuous
extension of $f$, that is also linear. Since $\boldsymbol{p}_{\kappa
^{-}}:X\rightarrow \boldsymbol{X}_{\kappa ^{-}}$ is a $\mathcal{B}_{\kappa
^{-}}$-expansion of $X$, there exist a $\lambda \in \Lambda _{\kappa ^{-}}$
and a unique continuous linear function $f^{\lambda }:X_{\lambda
}\rightarrow Z$ such that $f^{\lambda }p_{\lambda }=\bar{f}.$ Then $%
f^{\lambda }p_{\lambda }j=\bar{f}j=f$. If a $g^{\lambda }:X_{\lambda
}\rightarrow Z$ has the same property, i.e., $g^{\lambda }p_{\lambda }j=f$,
then $g^{\lambda }p_{\lambda }j=\bar{f}j$. Since every continuous extension
of $Z$ onto $X$ is unique, it follows that $g^{\lambda }p_{\lambda }=\bar{f}%
. $ Finally, the factorization trough an expansion term is also unique,
hence $g^{\lambda }=f^{\lambda }$.The dimension property holds because there
is no countably infinite-dimensional Banach space.
\end{proof}

\begin{lemma}
\label{L2}For every vectorial space $X\neq \{\theta \}$ over $F\in \{\mathbb{%
R},\mathbb{C}\}$,

\noindent (i) $\dim X_{0}^{\mathbb{N}}>\dim X\Leftrightarrow \dim X<\aleph
_{0}$,

\noindent where $X_{0}^{\mathbb{N}}$ is the direct sum space. Equivalently,

\noindent (i)' $\dim X_{0}^{\mathbb{N}}=\dim X\Leftrightarrow \dim X\geq
\aleph _{0}$,

\noindent and consequently,

\noindent (ii) $X\cong X_{0}^{\mathbb{N}}\cong F_{0}^{\mathbb{N}%
}\Leftrightarrow \dim X=\aleph _{0}$.

\noindent If, in addition, $X$ is a normed space and $Cl(X)$ is its
completion in the second dual space, then

\noindent (iii) $\dim Cl(X)>\dim X\Leftrightarrow \dim X=\aleph _{0}$.
\end{lemma}

\begin{proof}
Let $X\neq \{\theta \}$. If $\dim X<\aleph _{0}$, then $X\cong F^{n}$, for
some $n\in \mathbb{N}$, and the conclusion $\dim X_{0}^{\mathbb{N}}=\dim
F_{0}^{\mathbb{N}}=\aleph _{0}>n=\dim X$ follows straightforwardly.
Conversely, let $\dim X_{0}^{\mathbb{N}}>\dim X$. Let us assume to the
contrary, i.e., that $\dim X\geq \aleph _{0}$. If $\dim X=\aleph _{0}$, then
one readily sees that $\dim X_{0}^{\mathbb{N}}=\aleph _{0}=\dim X$. Thus, it
remains that $\dim X\geq 2^{\aleph _{0}}$, i.e., $\dim X=2^{\aleph _{k}}$, $%
k\geq 0$ (an ordinal). Then, by Lemma 3.2 (iv) of [13], $\dim X=|X|$, and
hence,

$2^{\aleph _{k}}=\dim X\leq \dim X_{0}^{\mathbb{N}}=|X_{0}^{\mathbb{N}}|$ $%
\leq |X^{\mathbb{N}}|$ $=|X|^{|\mathbb{N}|}=$

$=(2^{\aleph _{k}})^{\aleph _{0}}=2^{\aleph _{k}}=\dim X$

\noindent - a contradiction again. This proves equivalences (i) and (i)',
and the consequence (ii) follows. Let $X$ be a normed space such that $\dim
X=\aleph _{0}$. Then $X$ is not a Banach space. Thus its completion $Cl(X)$,
being a Banach space, must increase the algebraic dimension, i.e., $\dim
Cl(X)>\dim X$. Conversely, let $\dim Cl(X)>\dim X$. Assume to the contrary,
i.e., that either $\dim X<\aleph _{0}$ or $\dim X>\aleph _{0}$. If $\dim
X<\aleph _{0}$, then $X\cong F^{n}$ for some $n\in \mathbb{N}$, and hence, $%
Cl(X)=X$ - a contradiction. It remains that $\dim X>\aleph _{0}$, i.e., $%
\dim X=2^{\aleph _{k}}$, $k\geq 0$. Then

$2^{\aleph _{k}}=\dim X\leq \dim Cl(X)\leq |X^{\mathbb{N}}|$ $=|X|^{|\mathbb{%
N}|}=$

$=(2^{\aleph _{k}})^{\aleph _{0}}=2^{\aleph _{k}}=\dim X$

\noindent - a contradiction again, and equivalence (iii) is proven.
\end{proof}

Recall that every normed space is dense in its Banach completion. Further,
since the embedding into Banach completion is an isometry, all dense
subspaces of a normed space have the same Banach completion (in the second
dual space). Since there is no Banach space of the countably infinite
(algebraic) dimension, the following theorem is an immediate consequence of
Lemma 1 (see also Theorem 4 of [13]) and Lemma 2.

\begin{theorem}
\label{T1}(i) The quotient shape theory of

\noindent (i) $(\mathcal{N},\mathcal{N}_{\text{\b{0}}})$ ( and of $(\mathcal{%
N},\mathcal{B}_{\aleph _{0}})$ as well) reduces to that of $(\mathcal{B},%
\mathcal{B}_{\text{\b{0}}})$;

\noindent (ii) $(\mathcal{N},\mathcal{B}_{\kappa ^{-}})$ ($(\mathcal{N},%
\mathcal{B}_{\kappa })$) reduces to that of $(\mathcal{B},\mathcal{B}%
_{\kappa })$ ($(\mathcal{B},\mathcal{B}_{\kappa })$).
\end{theorem}

The following lemma is a generalization of [15], Proposition 1 (which was a
correction of incorrectly formulated [13], Corollary 4.4).

\begin{lemma}
\label{L3}Let $X$ and $Y$ be normed spaces over the same field such that $%
\dim X=\dim Y\equiv \kappa $. Then

\noindent (i) $Sh_{\text{\b{0}}}(X)=Sh_{\text{\b{0}}}(Y)$

\noindent and there exist an isomorphism $F:X\rightarrow Y$ of $Sh_{\text{\b{%
0}}}(\mathcal{N})$ that is induced by an isomorphism $\boldsymbol{f}^{\prime
}:\boldsymbol{X}^{\prime }\rightarrow \boldsymbol{Y}^{\prime }$ of $pro$-$%
\mathcal{H}_{\text{\b{0}}}$. If $\kappa >\aleph _{0}$, then

$Sh_{\aleph _{0}}(X)=Sh_{\aleph _{0}}(Y)$ with respect to Banach spaces,

\noindent and there exists an isomorphism $F^{\prime }:X\rightarrow Y$ of $%
Sh_{\aleph _{0}}(\mathcal{N})$ with respect to Banach spaces, that is
induced by the same isomorphism $\boldsymbol{f}^{\prime }:\boldsymbol{X}%
^{\prime }\rightarrow \boldsymbol{Y}^{\prime }$ of $pro$-$\mathcal{H}_{\text{%
\b{0}}}$.

\noindent (ii) If $\kappa >\aleph _{0}$ and, in addition, $X$ and $Y$ are
Banach spaces such that there exists a closed embedding $e:X\rightarrow Y$
so that $\dim (Y/e[X])<\kappa $, then

$Sh_{\kappa ^{-}}(X)=Sh_{\kappa ^{-}}(Y)$

\noindent and there exists an isomorphism $F:X\rightarrow Y$ of $Sh_{\kappa
^{-}}(\mathcal{B})$ that is induced by $e$.
\end{lemma}

\begin{proof}
(i). If $\kappa <\aleph _{0}$, then $X$ and $Y$ are isomorphic to an $F^{n}$%
, and thus, the statement is trivially true. Let $\kappa \geq \aleph _{0}$.
Clearly, one may assume that $X=(V,\left\Vert \cdot \right\Vert )$ and $%
Y=(V,\left\Vert \cdot \right\Vert ^{\prime })$, $\dim V=\kappa $. Let us
firstly construct a desired $\boldsymbol{f}:\boldsymbol{X}_{\text{\b{0}}%
}\rightarrow \boldsymbol{Y}_{\text{\b{0}}}$ of $pro$-$\mathcal{B}_{\text{\b{0%
}}}$. By a careful examining of the proof of [15], Proposition 1$,$ one
notices that there assumed continuity of $1_{V}:X\rightarrow Y$ does not
play any essential role. Namely, instead of by the identity $1_{V}$ induced
pro-morphism, one can construct a morphism

$\boldsymbol{f}=[(f,_{\mu })]:\boldsymbol{X}_{\text{\b{0}}}=(X_{\lambda
},p_{\lambda \lambda ^{\prime }},\Lambda _{\text{\b{0}}})\rightarrow (Y_{\mu
},q_{\mu \mu ^{\prime }},M_{\text{\b{0}}})=\boldsymbol{Y}_{\text{\b{0}}}$

\noindent of $pro$-$\mathcal{N}_{\text{\b{0}}}$ (actually, of $pro$-$%
\mathcal{B}_{\text{\b{0}}}$) in the same way as

$\boldsymbol{g}=[(g,g_{\lambda })]:\boldsymbol{Y}_{\text{\b{0}}}\rightarrow 
\boldsymbol{X}_{\text{\b{0}}}$,

\noindent (in that proof) is constructed. Mor precisely, by [9], Section 8.
11, (b), p. 440, every closed subspace $Z_{\lambda }\trianglelefteq X$ such
that $\dim Z_{\lambda }=\dim X=\dim V$ and $\dim (X/Z_{\lambda })<\aleph
_{0} $, induces a direct sum presentation $X=Z_{\lambda }\overset{\cdot }{+}$
$W_{\lambda }$, $W_{\lambda }\trianglelefteq X$ closed. Clearly, $W_{\lambda
}\cong X/Z_{\lambda }=X_{\lambda }$. And similarly, every closed subspace $%
Z_{\mu }\trianglelefteq Y$ such that $\dim Z_{\mu }=\dim Y=\dim V$ and $\dim
(Y/Z_{\mu })<\aleph _{0}$ induces a direct sum presentation $Y=Z_{\mu }%
\overset{\cdot }{+}$ $W_{\mu }$, $W_{\mu }\trianglelefteq Y$ closed, and $%
W_{\mu }\cong Y/Z_{\mu }=Y_{\mu }$. For each $\lambda \in \Lambda _{\text{\b{%
0}}}$ and each $\mu \in M_{\text{\b{0}}}$, choose and fix such a $W_{\lambda
}$ and a $W_{\mu }$ respectively. Recall that the morphisms $p_{\lambda
}:X\rightarrow X_{\lambda }$, $p_{\lambda \lambda ^{\prime }}:X_{\lambda
^{\prime }}\rightarrow X_{\lambda }$, $\lambda \leq \lambda ^{\prime }$, and 
$q_{\mu }:Y\rightarrow Y_{\mu }$, $q_{\mu \mu ^{\prime }}:X_{\mu ^{\prime
}}\rightarrow Y_{\mu }$, $\mu \leq \mu ^{\prime }$, are the corresponding
quotient functions, which all are linear and continuous. Observe that every
finite-dimensional subspace $W\trianglelefteq V$ is closed in the both $X$
and $Y$. Therefore, by [9], Section 8. 11, (c), p. 440, given a $\mu \in M_{%
\text{\b{0}}}$, i.e., a $Y_{\mu }$, for a chosen $W_{\mu }$, there exists a
closed subspace $Z_{\lambda _{\mu }}$ of $X$ that is a direct complement of $%
W_{\mu }$ and of the chosen $W_{\lambda _{\mu }}$ as well. Clearly, for
every $\mu \in M_{\text{\b{0}}}$,

$X_{\lambda _{\mu }}=X/Z_{\lambda _{\mu }}\cong W_{\lambda _{\mu }}\cong
W_{\mu }\cong Y/Z_{\mu }=Y_{\mu }$, \quad and

$Y=Z_{\mu }\overset{\cdot }{+}W_{\mu }$, $\quad X=Z_{\lambda _{\mu }}\overset%
{\cdot }{+}W_{\mu }$.

\noindent It implies that each $[v=y]_{\mu }\in Y_{\mu }$ is represented by
a unique $w\in W_{\mu }$ and conversely, and that each $[v=x]_{\lambda _{\mu
}}\in X_{\lambda _{\mu }}$ is represented by a unique $w^{\prime }\in W_{\mu
}$ and conversely.

\noindent Let us define

$\phi _{\mu }:W_{\mu }\rightarrow X_{\lambda _{\mu }}$, $\phi _{\mu
}(w)=[w]_{\lambda _{\mu }}=w+Z_{\lambda _{\mu }}$,

$\psi _{\mu }:W_{\mu }\rightarrow Y_{\mu }$, $\psi _{\mu }(w)=[w]_{\mu
}=w+Z_{\mu }$.

\noindent Since the elements of $W_{\mu }$ bijectively represent the
element-classes of $Y_{\mu }$ and of $X_{\lambda _{\mu }}$, it follows that $%
\phi _{\lambda }$ and $\psi _{\lambda }$ are linear bijections. Since all
the spaces are finite-dimensional, $\phi _{\lambda }$, $\phi _{\lambda
}^{-1} $, $\psi _{\lambda }$ and $\psi _{\lambda }^{-1}$ are continuous, and
thus, they are the isomorphisms of Banach spaces. Hence, the composite

$\phi _{\mu }\psi _{\mu }^{-1}:X_{\lambda _{\mu }}\rightarrow Y_{\mu }$ $,$ $%
\phi _{\mu }\psi _{\mu }^{-1}([x=w]_{\lambda _{\mu }})=[w=y]_{\mu }$,

\noindent is an isomorphism of Banach spaces. Put

$f:M_{\text{\b{0}}}\rightarrow \Lambda _{\text{\b{0}}}$, $f(\mu )=$ $\lambda
_{\mu },$ \quad and

$f_{\mu }:X_{f(\mu )}\rightarrow Y_{\mu }$, $f_{\mu }=\phi _{\mu }\psi _{\mu
}^{-1}$.

\noindent Then

$(f,f_{\mu }):\boldsymbol{X}_{\text{\b{0}}}\rightarrow \boldsymbol{Y}_{\text{%
\b{0}}}$

\noindent is a morphism of $inv$-$\mathcal{B}_{\text{\b{0}}}\subseteq inv$-$%
\mathcal{N}_{\text{\b{0}}}$. Indeed, for every related pair $\mu \leq \mu
^{\prime }$, i.e., $Z_{\mu ^{\prime }}\trianglelefteq Z_{\mu }$, there
exists a $\lambda \geq f(\mu ),f(\mu ^{\prime })$, and since each $%
[x]_{\lambda }=[w]_{\lambda }\in X_{\lambda }$ for one and only one $w\in
W_{\lambda ,}$, it follows that

$q_{\mu \mu ^{\prime }}f_{\mu ^{\prime }}p_{f(\mu ^{\prime })\lambda
}([x=w]_{\lambda })=q_{\mu \mu ^{\prime }}f_{\mu ^{\prime }}([w]_{f(\mu
^{\prime })})=q_{\mu \mu ^{\prime }}\phi _{\mu ^{\prime }}\psi _{\mu
^{\prime }}^{-1}([w]_{f(\mu ^{\prime })})=$

$=q_{\mu \mu ^{\prime }}([w]_{\mu ^{\prime }})=[w=y]_{\mu }$, \quad and

$f_{\mu }p_{f(\mu )\lambda }([x=w]_{\lambda })=f_{\mu }([w]_{f(\mu )})=$ $%
\phi _{\mu }\psi _{\mu }^{-1}([w]_{f(\mu )})=[w=y]_{\mu }$.

\noindent Denote by

$\boldsymbol{f}=[(f,f_{\mu }):\boldsymbol{X}_{\text{\b{0}}}\rightarrow 
\boldsymbol{Y}_{\text{\b{0}}}$

\noindent the induced morphism of $pro$-$\mathcal{N}_{\text{\b{0}}}=pro$-$%
\mathcal{B}_{\text{\b{0}}}$. One can now construct, in the same way, a
morphism

$\boldsymbol{g}=[(g,g_{\lambda }):\boldsymbol{Y}_{\text{\b{0}}}\rightarrow 
\boldsymbol{X}_{\text{\b{0}}}$

\noindent and straightforwardly prove that $\boldsymbol{g}^{-1}=\boldsymbol{f%
}$. However, it is more convenient to observe that the index function $f:M_{%
\text{\b{0}}}\rightarrow \Lambda _{\text{\b{0}}}$ is cofinal, i.e., that
every $\lambda $ admits a $\mu $ such that $f(\mu )\geq \lambda $. Namely,
one readily sees that $X=Z_{\lambda }\overset{\cdot }{+}$ $W_{\lambda }$
admits

$Y=Z_{\mu _{\lambda }}\overset{\cdot }{+}$ $W_{\lambda }=(Z_{\mu }\overset{%
\cdot }{+}W)\overset{\cdot }{+}W_{\lambda }=Z_{\mu }\overset{\cdot }{+}(W%
\overset{\cdot }{+}$ $W_{\lambda })\equiv Z_{\mu }\overset{\cdot }{+}$ $%
W_{\mu }$

\noindent such that $Z_{\mu }$ is closed in $Y$, $\dim Z_{\mu }=\dim Y$ and $%
\dim W<\aleph _{0}$. Then, by the canonical construction and the definition
of $f$, it follows that $f(\mu )=\lambda _{\mu }\geq \lambda $. Now, the
conclusion that $\boldsymbol{f}:\boldsymbol{X}_{\text{\b{0}}}\rightarrow 
\boldsymbol{Y}_{\text{\b{0}}}$ is an isomorphism follows by the fact that
each its term $f_{\mu }:X_{\phi (\mu )}\rightarrow Y_{\mu }$, $\mu \in M_{%
\text{\b{9}}}$, is an isomorphisms of Banach spaces. Then $F=\left\langle 
\boldsymbol{f}\right\rangle \in Sh_{\text{\b{0}}}(X,Y)$ is a desired finite
quotient shape isomorphism. If, in addition, $\dim V>\aleph _{0}$, i.e. ($CH$
accepted), $\dim V\geq 2^{\aleph _{0}}$, then, by Lemma 2 (iii), $\dim
Cl(X)=\dim X=\dim V=\dim Y=\dim Cl(Y)\geq 2^{\aleph _{0}}$, and the
conclusion about the countable quotient shapes and the isomorphism $%
F^{\prime }=\left\langle \boldsymbol{f}\right\rangle ^{\prime }:X\rightarrow
Y$ of $Sh_{\aleph _{0}}(\mathcal{N})$ follows by Theorem 1 and Lemma 1. Let
us now find a representative of the quotient shape isomorphism $%
F:X\rightarrow Y$ of $Sh_{\text{\b{0}}}(\mathcal{N})$ ($F^{\prime
}:X\rightarrow Y$ of $Sh_{\aleph _{0}}(\mathcal{N})$ with respect to Banach
spaces, whenever $\kappa >\aleph _{0}$) belonging to $pro$-$\mathcal{H}_{%
\text{\b{0}}}$. If $\dim V<\aleph _{0}$, the canonical rudimentary identity
\textquotedblleft expansions\textquotedblright\ may be replaced by the
isomorphic ones with the Hilbert codomains $F^{n}(2)$, where $n=\dim V\in 
\mathbb{N}$. Let $\dim V>\aleph _{0}$. Firstly, we are to construct an
inverse system

$\boldsymbol{X}^{\prime }=(X_{\lambda }^{\prime },p_{\lambda \lambda
^{\prime }}^{\prime },\Lambda ^{\prime })$

\noindent in $\mathcal{H}_{\text{\b{0}}}$, and an isomorphism

$\boldsymbol{u}:\boldsymbol{X}_{\text{\b{0}}}\rightarrow \boldsymbol{X}%
^{\prime }$

\noindent of $pro$-$\mathcal{B}_{\text{\b{0}}}$. By [10], I.1.2, Theorem 2,
we may assume, without loss of generality, that $\Lambda _{\text{\b{0}}}$ is
cofinite (every $\lambda \in \Lambda _{\text{\b{0}}}$ admits at most
finitely many predecessors). Then the construction goes by induction on $%
|\lambda |$ $\in \{0\}\cup \mathbb{N}$. Let $|\lambda |$ $=0$. Then $%
X_{\lambda }\cong F$ and, by the canonical construction, no pair $\lambda
,\lambda ^{\prime }$ is related whenever $|\lambda |$ $=|\lambda ^{\prime }|$
$=0$. Put $X_{\lambda }^{\prime }=F$ and $p_{\lambda \lambda }^{\prime
}=1_{X_{\lambda }^{\prime }}$, and choose an isomorphism $u_{\lambda
}:X_{\lambda }\rightarrow X_{\lambda }^{\prime }$. Let $n\in \mathbb{N}$,
and assume that, for all $\lambda \in \Lambda _{\text{\b{0}}}$ such that $%
|\lambda |$ $<n$, the construction is made, i.e., for all $\lambda ^{\prime
}\leq \lambda $, the Hilbert spaces $X_{\lambda ^{\prime }}^{\prime }$ and
the isomorphisms $u_{\lambda ^{\prime }}:X_{\lambda ^{\prime }}\rightarrow
X_{\lambda ^{\prime }}^{\prime }$ are chosen and, for every related pair $%
\lambda _{1}\leq \lambda _{2}$ $(\leq \lambda )$ and every related triple $%
\lambda _{1}\leq \lambda _{2}\leq \lambda _{3}$ $(\leq \lambda )$ the bonds $%
p_{\lambda _{1}\lambda _{2}}^{\prime }$, $p_{\lambda _{1}\lambda
_{3}}^{\prime }$, $p_{\lambda _{2}\lambda _{3}}^{\prime }$ are defined
according to commutativity conditions $p_{\lambda _{1}\lambda _{2}}^{\prime
}u_{\lambda _{2}}=u_{\lambda _{1}}p_{\lambda _{1}\lambda _{2}}$ and $%
p_{\lambda _{1}\lambda _{2}}^{\prime }p_{\lambda _{2}\lambda _{3}}^{\prime
}=p_{\lambda _{1}\lambda _{3}}^{\prime }$. Let $\lambda \in \Lambda _{\text{%
\b{0}}}$ such that $|\lambda |$ $=n\in \mathbb{N}$, and let $\lambda
_{1},\ldots ,\lambda _{n}\lvertneqq \lambda $ be all the predecessors of $%
\lambda $. Then, for each $i=1,\ldots ,n$, $|\lambda _{i}|$ $<n$ holds.
Thus, by the inductive assumption, for all $\lambda _{i}$ and all their
predecessors the construction is already made. By the canonical construction
of a quotient expansion, $X_{\lambda }\cong F^{k(\lambda )}$, $k(\lambda
)\in \mathbb{N}$. Put $X_{\lambda }^{\prime }=F^{k(\lambda )}$ and $%
p_{\lambda \lambda }^{\prime }=1_{X_{\lambda }^{\prime }}$, and choose an
isomorphism $u_{\lambda }:X_{\lambda }\rightarrow X_{\lambda }^{\prime }$.
Now define, for each $i=1,\ldots ,n$, $p_{\lambda _{i}\lambda }^{\prime
}=u_{\lambda _{i}}p_{\lambda _{i}\lambda }u_{\lambda }^{-1}$. A
straightforward verification shows that, for every $i$, $p_{\lambda
_{i}\lambda }^{\prime }u_{\lambda }=u_{\lambda _{i}}p_{\lambda _{i}\lambda }$%
, and that, for every related pair $\lambda _{i^{\prime }}\leq \lambda
_{i^{\prime }}$, $p_{\lambda _{i}\lambda _{i^{\prime }}}^{\prime }p_{\lambda
_{i^{\prime }}\lambda }^{\prime }=p_{\lambda _{i}\lambda }^{\prime }$ holds.
This completes the inductive construction of an inverse system $\boldsymbol{X%
}^{\prime }\in Ob(pro$-$\mathcal{H}_{\text{\b{0}}})$ and a pro-morphism $%
\boldsymbol{u}=[[1_{\Lambda _{\text{\b{0}}}},u_{\lambda })]:\boldsymbol{X}_{%
\text{\b{0}}}\rightarrow \boldsymbol{X}^{\prime }$

\noindent (of $pro$-$\mathcal{B}_{\text{\b{0}}}$). Since the index function $%
!_{\Lambda _{\text{\b{0}}}}$ is, obviously, cofinal and each $u_{\lambda }$
is an isomorphism of the Banach spaces$,$ it follows that $\boldsymbol{u}$
is an isomorphism of $pro$-$\mathcal{B}_{\text{\b{0}}}$. In the same way one
can construct an isomorphism $\boldsymbol{v}:\boldsymbol{Y}_{\text{\b{0}}%
}\rightarrow \boldsymbol{Y}^{\prime }$ of $pro$-$\mathcal{B}_{\text{\b{0}}}$%
, where $\boldsymbol{Y}^{\prime }\in Ob(pro$-$\mathcal{H}_{\text{\b{0}}})$.
Clearly,

$\boldsymbol{up}_{\text{\b{0}}}:X\rightarrow \boldsymbol{X}^{\prime }$ \quad
and

$\boldsymbol{vq}_{\text{\b{0}}}:Y\rightarrow \boldsymbol{Y}^{\prime }$

\noindent are $\mathcal{B}_{\text{\b{0}}}$-expansions of $X$ and $Y$,
respectively, having the expansion systems in $\mathcal{H}_{\text{\b{0}}}$.
The proof of the first statement of (i) is complete by putting $\boldsymbol{f%
}^{\prime }=\boldsymbol{vfu}^{-1}$. If $\kappa >\aleph _{0}$, then the
countable quotient shape with respect to Banach spaces reduces to the finite
one (Theorem 1 (i)), and the conclusion follows as previously.

\noindent (ii). We may assume that $X$ is a closed subspace of $Y$ such that 
$\dim X=\dim Y=\kappa \geq 2^{\aleph _{0}}$ and $\dim (Y/X)<\kappa $. Recall
that $\kappa =2^{\aleph _{0}}$ implies $Sh_{\kappa ^{-}}(\mathcal{B})=Sh_{%
\text{\b{0}}}(\mathcal{B})$. Let

$\boldsymbol{p}_{\kappa ^{-}}=(p_{\lambda }):X\rightarrow \boldsymbol{X}%
_{\kappa ^{-}}=(X_{\lambda },p_{\lambda \lambda ^{\prime }},\Lambda _{\kappa
^{-}})$,

$\boldsymbol{q}_{\kappa ^{--}}=(q_{\mu }):Y\rightarrow \boldsymbol{Y}%
_{\kappa ^{-}}=(Y_{\mu },q_{\mu \mu ^{\prime }},M_{\kappa ^{-}})$

\noindent be the canonical $\mathcal{B}_{\kappa ^{-}}$-expansions of $X$, $Y$
respectively. Let

$(\varphi ,i_{\mu }):\boldsymbol{X}_{\kappa ^{-}}\rightarrow \boldsymbol{Y}%
_{\kappa ^{-}}$

\noindent be by the inclusion $i:X\hookrightarrow Y$ induced morphism of $%
inv $-$\mathcal{B}_{\kappa ^{.}}$. By the canonical construction of these $%
\kappa ^{-}$-expansions, if the index function $\varphi :M_{\kappa
^{-}}\rightarrow \Lambda _{\kappa ^{-}}$ is cofinal (i.e., if each $\lambda
\in \Lambda _{\kappa ^{-}}$ admits a $\mu \in M_{\kappa ^{-}}$ such that $%
\varphi (\mu )\geq \lambda $), then the equivalence class

$[(\varphi ,i_{\mu })]:\boldsymbol{X}_{\kappa ^{-}}\rightarrow \boldsymbol{Y}%
_{\kappa ^{-}}$

\noindent is an isomorphism of $pro$-$\mathcal{B}_{\kappa ^{.}}$, implying $%
Sh_{\kappa ^{-}}(X)=Sh_{\kappa ^{-}}(Y)$. Therefore, according to the
previous case, the proof reduces to the verification of the following claim:

\emph{For every closed subspace }$W\trianglelefteq Y$\emph{\ such that }$%
\dim W=\dim Y$\emph{\ and }$\dim (Y/W)<\kappa $\emph{, there exists a closed
subspace }$Z\trianglelefteq X$\emph{\ such that }$\dim Z=\dim X$\emph{\ and }%
$\dim (X/Z)<\kappa $\emph{, and in addition, }$Z\trianglelefteq W$\emph{, }$%
\dim Z=\dim Y$\emph{\ and }$\dim (Y/Z)<\kappa $\emph{. }

\noindent Let such a $W$ be given. Put $Z=X\cap W$. Then $Z\trianglelefteq W$
is a closed subspace of the both $X$ and $Y$. Since the codimensions of $X$
and $W$ are less than $\kappa $, it follows (see also Lemma 3.8 (iii) of
[13]) that

$\dim Z=\dim (X\cap W)=\dim X=\dim Y=\kappa $.

\noindent Further, since, in addition, $\dim (Y/X)<\kappa $ and $\dim
(Y/W)<\kappa ,$ where $\kappa =2^{\aleph _{k}}$, for some ordinal $k\geq 1$ (%
$GCH$ accepted), it follows that

$\dim (X/Z)=\dim (X/(X\cap W))\leq \dim (Y/(X\cap W))=\dim (Y/Z)<\kappa $.

\noindent So the claim is verified, and the proof is completed.
\end{proof}

Given a vectorial space $V$, denote

$\mathcal{W}(V)=\{W\trianglelefteq V\mid \dim W<\aleph _{0}\}$,

$\mathcal{Z}(V)=\{Z\trianglelefteq X\mid \dim Z=\dim X\wedge \dim
(X/Z)<\aleph _{0}\}$.

\noindent Further, given a normed space $X$, denote

$\mathcal{Z}_{cl}(X)=\{Z\trianglelefteq X\mid Cl(Z)=Z\curlywedge \dim Z=\dim
X\wedge \dim (X/Z)<\aleph _{0}\}$.

We shall need the following general facts.

\begin{lemma}
\label{L4}Let $V$ be a vectorial spaces over $F\in \{\mathbb{R},\mathbb{C}\}$%
. If $\dim V>\aleph _{0}$, then

\noindent (i) $|\mathcal{W}(V)|$ $=|V|$ $=\dim V$ $<2^{|V|}=|\mathcal{Z}(V)|$%
.

\noindent Further, if $X=(V,\left\Vert \cdot \right\Vert )$ is a separable
or bidual-like normed space ($X^{\ast \ast }\cong X$), then

\noindent (ii) $|\mathcal{W}(X)|$ $=|\mathcal{Z}_{cl}(X)|$ $=|X|$ $=\dim X$.
\end{lemma}

\begin{proof}
The equality $|V|$ $=\dim V$, and thus $|X|$ $=\dim X$ as well, follows by
[13], Lemma 3.2 (iv). Notice that $\mathcal{W}(V)$ is the disjoint union of
all $\mathcal{W}_{n}(V)$, $n\in \{0\}\cup \mathbb{N}$, where

$\mathcal{W}_{n}(V)=\{W\trianglelefteq V\mid \dim W=n\}$.

\noindent Hereby, $\mathcal{W}_{0}(V)=\{\theta \}$. The same holds for $%
\mathcal{W}(X)$. (Recall that every finite-dimensional subspace $W$ of $X$
is closed. ) Similarly, $\mathcal{Z}(V)$ is the disjoint union of all $%
\mathcal{Z}_{n}(V)$, $n\in \{0\}\cup \mathbb{N}$, where

$\mathcal{Z}_{n}(V)=\{Z\trianglelefteq V\mid \dim Z=\dim V\wedge \dim
(V/Z)=n\}$.

\noindent Hereby $\mathcal{Z}_{0}(V)=\{V\}$. In the same way, $\mathcal{Z}%
_{cl}(X)$ is the disjoint union of all $\mathcal{Z}_{cl,n}(X)$, $n\in
\{0\}\cup \mathbb{N}$, where

$\mathcal{Z}_{c\ln ,}(X)=\{Z\trianglelefteq X\mid Cl(Z)=Z\curlywedge \dim
Z=\dim X\wedge \dim (X/Z)=n\}$

\noindent and $\mathcal{Z}_{cl,0}(X)=\{X\}$. Observe that $|\mathcal{W}%
_{1}(V)|$ $=|V|$, and $|\mathcal{W}_{1}(X)|$ $=|X|$ as well, hold because of 
$|V|$ $\geq 2^{\aleph _{0}}$. Further, one readily sees that, for every $%
n\neq 0$, $\dim V>n$ implies $|\mathcal{W}_{n}(V)|$ $=|\mathcal{W}_{1}(V)|$ $%
\geq 2^{\aleph _{0}}$. The same holds true for $X$. Now observe that

$|\mathcal{Z}_{0}(V)\cup \mathcal{Z}_{1}(V)|$ $=|V^{+}|$,

\noindent where $V^{+}$ denotes the (algebraic) dual of $V,$ while

$|\mathcal{Z}_{0}(X)\cup \mathcal{Z}_{1}(X)|$ $=|X^{\ast }|$,

\noindent where $X^{\ast }$ denotes the (normed) dual of $X$. Then, it is
easy to see that, for every $n\neq 0$, $|\mathcal{Z}_{n}(V)|$ $=|\mathcal{Z}%
_{1}(V)|$ and $|\mathcal{Z}_{cl,n}(X)|$ $=|\mathcal{Z}_{cl,1}(X)|$.
Consequently, statement (i) reduces to

$|\mathcal{W}_{1}(V)|$ $=|V|$ $=\dim V$ $<2^{|V|}=|V^{+}|$ $=|\mathcal{Z}%
_{1}(V)|$,

\noindent that holds true because of $\dim V\geq 2^{\aleph _{0}}$.
Similarly, concerning (ii), it suffices to prove that

$|\mathcal{W}_{1}(X)|$ $=|X|$ $=|X^{\ast }|$ $=|\mathcal{Z}_{cl,1}(X)|$,

\noindent where the first and third equality hold already. It remains to
prove the second one. Clearly, $|X|$ $\leq |X^{\ast }|$. Assume, firstly,
that $X$ is a separable normed space. Then $|X|$ $=2^{\aleph _{0}}$, while
the cardinality of the set $F^{X}$ of all functions of $X$ to $F$ is $%
|F|^{|X|}=2^{\aleph _{1}}$ ($GCH$ accepted). Since $X$ is separable, the
cardinality of the set $c(X,F)$ of all continuous functions of $X$ to $F$ is
determined by \emph{countability} of a dense subset on $X$. This implies
that $|c(X,F)|$ $=2^{\aleph _{0}}=|X|$. Since $|X^{\ast }|$ $\leq |c(X,F)|$,
the conclusion follows. Finally, if $X$ is a bidual-like normed space, i.e., 
$X^{\ast \ast }\cong X$, then $|X^{\ast \ast }|$ $=|X|$, implying $|X^{\ast
}|$ $=|X|$.
\end{proof}

\begin{remark}
\label{R1}If $\dim X=\aleph _{0}$, i.e., $X\cong (F_{0}^{\mathbb{N}%
},\left\Vert \cdot \right\Vert )$ (the direct sum), then $the$ cardinalities
considered in Lemma 4 are $|\mathcal{W}(F_{0}^{\mathbb{N}})|$ $=|\mathcal{Z}%
_{cl}(F_{0}^{\mathbb{N}})|$ $=2^{\aleph _{0}}=|F_{0}^{\mathbb{N}}|$ $>\dim X$%
. Further, in all finite-dimensional cases, i.e., $X\cong F^{n}$, $n\in 
\mathbb{N}$, and $|\mathcal{W}(F^{n})|$ $=|F^{n}|$ $=|F|$ $=2^{\aleph _{0}}$%
, while $|\mathcal{Z}_{cl}(F^{n})|$ $=1$ because $\mathcal{Z}=\{X\}$. Notice
that, for the \emph{full} subcategory $\mathcal{U}\subseteq \mathcal{N}$ ($U$
- unitary; \emph{all} the continuous linear functions included), it holds $%
X\cong Y$ in $\mathcal{U}$ if and only if $X\cong Y$ in $\mathcal{N}$.
Namely, though the restriction to the linear inner-product preserving
functions is convenient for making quotient spaces, it breaks the shape
relationship with $\mathcal{N}$ (hereby a full subcategory is needed!). In
other words, the restriction to the inner product preserving morphisms,
would lead to a \emph{new} quotient shape theory of unitary (Hilbert) spaces
and linear inner-product preserving functions.
\end{remark}

Though the separability assumption of a non-bidual-like space $X$ in our
proof of Lemma 4 (ii) is essential, the following question still makes sense:

\textbf{Question 1.} \emph{Does Lemma 4 (ii) hold true for every normed
(Banach) space }$X$\emph{\ having }$\dim X\neq \aleph _{0}?$

Namely, the author can prove that $X=l_{\infty }$ (non-separable and
non-bidual-like) is an example towards the affirmative answer.

\begin{theorem}
\label{T2}The finite quotient shape type of normed spaces over the same
field is a strict invariant of the (algebraic) dimension, i.e.,

\noindent (i) $\quad (\dim X=\dim Y)\Rightarrow (Sh_{\text{\b{0}}}(X)=Sh_{%
\text{\b{0}}}(Y))$.

\noindent (and with respect to $\mathcal{B}_{\text{\b{0}}}$ as well).
Furthermore, there exists a quotient shape isomorphism $F=\left\langle 
\boldsymbol{f}\right\rangle :X\rightarrow Y$ of $Sh_{\text{\b{0}}}(NVectlF)$
induced by an isomorphism $\boldsymbol{f}$ of $pro$-$\mathcal{H}_{\text{\b{0}%
}}$.

\noindent Further, either $\max \{\dim X,\dim Y\}\leq \aleph _{0}$ or $%
X,Y\in Ob(s\mathcal{B}\cup b\mathcal{B})$ such that $\min \{\dim X,\dim
Y\}>\aleph _{0},$ then

\noindent (ii) $\quad (\dim X=\dim Y)\Leftrightarrow (Sh_{\text{\b{0}}%
}(X)=Sh_{\text{\b{0}}}(Y))\Leftrightarrow (Sh_{\aleph _{0}}(X)=Sh_{\aleph
_{0}}(Y))$

\noindent hold true. Consequently, the classifications on $s\mathcal{B}\cup b%
\mathcal{B}$ by (algebraic) dimension, by the finite quotient shape $and$ by
the countable quotient shape coincide.

\noindent Furthermore, for the bidual-like Banach spaces, if there exists a
closed embedding $e:X\rightarrow Y$ such that $\dim (Y/e[X])<\dim Y=\kappa $%
, then

\noindent (iii) $(\dim X=\dim Y)\Leftrightarrow (Sh_{\kappa
^{-}}(X)=Sh_{\kappa ^{-}}(Y))$.
\end{theorem}

\begin{proof}
Let $X$ and $Y$ be normed spaces over the same field such that $\dim X=\dim
Y $. Since one may assume that $X=(V,\left\Vert \cdot \right\Vert )$ and $%
Y=(V,\left\Vert \cdot \right\Vert ^{\prime })$, the implication (i), and the
necessity parts in (ii) follow by Theorem 1 and Lemma 3, while the second
sufficiency in (ii) holds trivially.

In order to prove that the converse of (i) does not hold, let us consider
the direct sum vectorial (algebraic) space $F_{0}^{\mathbb{N}}$ ($%
\trianglelefteq l_{p}\trianglelefteq F^{\mathbb{N}}$) and the corresponding
normed subspaces $F_{0}^{\mathbb{N}}(p)$ (of $l_{p}$), $1\leq p\leq \infty $%
, that all are of dimension $\dim F_{0}^{\mathbb{N}}=\aleph _{0}$. Since $%
Cl(F_{0}^{\mathbb{N}}(p))=l_{p}$, $1\leq p<\infty $, and $Cl(F_{0}^{\mathbb{N%
}}(\infty ))=l_{p}(\infty )$ in $l_{p}(\infty )$, and $Cl(F_{0}^{\mathbb{N}%
}(\infty ))=c_{0}$ in $l_{\infty }$ (see also [15], Section 4), and since $%
\dim l_{p}=\dim c_{0}=2^{\aleph _{0}}>\dim (F_{0}^{\mathbb{N}}(p))$ and $Sh_{%
\text{\b{0}}}(l_{p})=Sh_{\text{\b{0}}}(F_{0}^{\mathbb{N}}(p))=Sh_{\text{\b{0}%
}}(c_{0})$ (Lemma 1), it follows that there is a lot of counterexamples in
the dimensional pair $\{\aleph _{0},2^{\aleph _{0}}\}$.

\noindent Let us now prove the sufficiency part of the first equivalence in
(ii), Let $X$, $Y$ be a pair of normed spaces over the same field such that
either the both $\dim X,\dim Y\leq \aleph _{0}$ or $X,Y\in Ob(s\mathcal{N}%
\cup b\mathcal{N})$ having $\dim X,\dim Y\geq 2^{\aleph _{0}}$ ($CH$
accepted), and let us assume that $Sh_{\text{\b{0}}}(X)=Sh_{\text{\b{0}}}(Y)$%
. If $\dim X<\aleph _{0}$ and $\dim Y<\aleph _{0}$, then the both $X$ and $Y$
have to be isomorphic to an $F^{n}$, $n\in \mathbb{N}$, and hence, $\dim
X=\dim Y$. Further, $Sh_{\text{\b{0}}}(X)=Sh_{\text{\b{0}}}(Y)$ and $n=\dim
X<\dim Y=\aleph _{0}$ (or $n=\dim Y<\dim X=\aleph _{0}$) immediately leads
to a contradiction. It remains to prove the statement in the case of $X,Y\in
Ob(s\mathcal{N}\cup s\mathcal{N})$ having $\dim X\geq 2^{\aleph _{0}}$ and $%
\dim Y\geq 2^{\aleph _{0}}$. Let

$\boldsymbol{p}_{\text{\b{0}}}=(p_{\lambda }):X\rightarrow \boldsymbol{X}_{%
\text{\b{0}}}=(X_{\lambda },p_{\lambda \lambda ^{\prime }},\Lambda _{\text{%
\b{0}}})$,

$\boldsymbol{q}_{\text{\b{0}}^{-}}=(q_{\mu }):Y\rightarrow \boldsymbol{Y}_{%
\text{\b{0}}}=(Y_{\mu },q_{\mu \mu ^{\prime }},M_{\text{\b{0}}})$

\noindent be the canonical $\mathcal{N}_{\text{\b{0}}}$-expansions of $X$, $%
Y $ respectively, that also are the $\mathcal{B}_{\text{\b{0}}}$-expansions.
Then $\boldsymbol{X}_{\text{\b{0}}}\cong \boldsymbol{Y}_{\text{\b{0}}}$ in $%
pro$.$\mathcal{B}_{\text{\b{0}}}$. According to [10], I.1.2, Theorem 2, we
may pass to the associated cofinite (i.e., with cofinite index sets) inverse
systems

$\boldsymbol{X}^{\prime }=(X_{\bar{\lambda}}^{\prime },p_{\bar{\lambda}\bar{%
\lambda}^{\prime }}^{\prime },\bar{\Lambda})$ \quad and

$\boldsymbol{Y}^{\prime }=(Y_{\bar{\mu}}^{\prime },q_{\bar{\mu}\bar{\mu}%
}^{\prime },\bar{M})$

\noindent such that $|\bar{\Lambda}|$ $=|\Lambda _{\text{\b{0}}}|$, $|\bar{M}%
|$ $=|M_{\text{\b{0}}}|$ and $\boldsymbol{X}^{\prime }\cong \boldsymbol{X}_{%
\text{\b{0}}}$, $\boldsymbol{Y}^{\prime }\cong \boldsymbol{Y}_{\text{\b{0}}}$
in $pro$-$\mathcal{B}_{\text{\b{0}}}$. Then $\boldsymbol{X}^{\prime }\cong 
\boldsymbol{Y}^{\prime }$ in $pro$-$\mathcal{B}_{\text{\b{0}}}$. Let $%
\boldsymbol{f}:\boldsymbol{X}^{\prime }\rightarrow \boldsymbol{Y}^{\prime }$
be an isomorphism. Choose a special representative $(\phi ,f_{\bar{\mu}})$
of $\boldsymbol{f}$ (having the index function $\phi $ increasing, [10], I.
1.2, Lemma 3). Then, for every related pair $\bar{\mu}\leq \bar{\mu}^{\prime
}$,

$f_{\mu }p_{\phi (\bar{\mu})\phi (\bar{\mu}^{\prime })}^{\prime }=q_{\bar{\mu%
}\bar{\mu}^{\prime }}^{\prime }f_{\bar{\mu}^{\prime }}$.

\noindent Similarly, there exists a special representative $(\psi ,g_{\bar{%
\lambda}})$ of $\boldsymbol{f}^{-1}:\boldsymbol{Y}^{\prime }\rightarrow 
\boldsymbol{X}^{\prime }$ ($\psi $ is increasing) such that, for every
related pair $\bar{\lambda}\leq \bar{\lambda}^{\prime }$,

$g_{\bar{\lambda}}q_{\psi (\bar{\lambda})\psi (\bar{\lambda}^{\prime \prime
})}^{\prime }=p_{\bar{\lambda}\bar{\lambda}^{\prime }}^{\prime }g_{\bar{\mu}%
^{\prime }}$.

\noindent Since

$(\phi ,f_{\bar{\mu}})\circ (\psi ,g_{\bar{\lambda}})=(\psi \phi ,f_{\bar{\mu%
}}g_{\phi (\bar{\mu})})\sim (1_{M_{\kappa ^{-}}},1_{Y_{\bar{\mu}}^{\prime
}}) $,

\noindent we conclude that, for every $\bar{\mu}\in \bar{M}$, there exists a 
$\bar{\mu}^{\prime }\geq \bar{\mu},\psi \phi (\bar{\mu})$ such that

$f_{\bar{\mu}}g_{\phi (\bar{\mu})}q_{\psi \phi (\bar{\mu})\bar{\mu}^{\prime
}}^{\prime }=q_{\bar{\mu}\bar{\mu}^{\prime }}^{\prime }$.

\noindent Recall that all $q_{\mu \mu ^{\prime }}$ are epimorphisms, and
thus such are all $q_{\bar{\mu}_{1}\bar{\mu}_{2}}^{\prime }$ ($q_{\bar{\mu}%
\bar{\mu}^{\prime }}^{\prime }=$

\noindent $q_{\bar{\mu}\psi \phi (\bar{\mu})}^{\prime }q_{\psi \phi (\bar{\mu%
})\bar{\mu}^{\prime }}^{\prime }$, by the choice of the special
representatives). Therefore,

$(\forall \bar{\mu}\in M_{\kappa ^{-}})\quad f_{\bar{\mu}}g_{\phi (\bar{\mu}%
)}=q_{\bar{\mu}\psi \phi (\mu )}^{\prime }$.

\noindent Now, assume to the contrary, i.e., that $\dim X\neq \dim Y$. We
may assume, without loss of generality, that $\dim X=\kappa \geq 2^{\aleph
_{0}}$, $\dim Y=\kappa ^{\prime }>\kappa $, i.e. ($GCH$ accepted), that $%
\kappa ^{\prime }\geq 2^{\kappa }$. Then, by Lemma 4 (ii), $|\Lambda _{\text{%
\b{0}}}$\TEXTsymbol{\vert} $=|\mathcal{Z}_{cl}(X)|$ $=\kappa $ and $|M_{%
\text{\b{0}}}|$ $=|\mathcal{Z}_{cl}(Y)$ $=|\kappa ^{\prime }$. Since

$\left\vert \bar{M}\right\vert =\left\vert M_{\text{\b{0}}}\right\vert
=\kappa ^{\prime }\geq 2^{\kappa }>\kappa =\left\vert \Lambda _{\text{\b{0}}%
}\right\vert =\left\vert \bar{\Lambda}\right\vert \geq 2^{\aleph _{0}}$,

\noindent there exists a $\bar{\lambda}_{0}\in \bar{\Lambda}$ and there
exists infinitely many (actually, $\kappa ^{\prime }$ many) elements $\bar{%
\mu}\in \bar{M}$ such that $\phi (\bar{\mu})=\bar{\lambda}_{0}$. Put $\bar{%
\mu}_{0}=\psi (\bar{\lambda}_{0})\in \bar{M}$. Then,

$\bar{\mu}_{0}=\psi \phi (\bar{\mu})\geq \bar{\mu}$

\noindent for infinitely many $\bar{\mu}\in \bar{M}$. It follows that $\bar{M%
}$ is not cofinite - a contradiction. Finally, the statement concerning
separable Banach spaces follows by Theorem 1 (i).

\noindent It remains to prove the last statement. Since $Sh_{\kappa
^{-}}(X)=Sh_{\kappa ^{-}}(Y)$ implies $Sh_{\text{\b{0}}}(X)=Sh_{\text{\b{0}}%
}(Y)$, the sufficiency part of (iii) holds by the sufficiency part of (ii)
in general, i.e., without any additional assumption. Conversely, let $X,Y\in
Obb\mathcal{B}$ such that $\dim X=\dim Y\equiv \kappa $ and let there exist
a closed embedding $e:X\rightarrow Y$ such that $\dim (Y/e[X])<\kappa $. If $%
\kappa \leq 2^{\aleph _{0}}$, then the $\kappa ^{-}$-quotient shape reduces
to \b{0}-quotient shape, and the conclusion follows, in general, by the
necessity part of (ii). Finally, in the case of $\kappa >2^{\aleph _{0}}$,
the conclusion follows by Lemma 3 (ii).
\end{proof}

\begin{corollary}
\label{C1}All infinite-dimensional separable normed spaces over $F$
(especially, all the direct sum spaces $(F_{0}^{\mathbb{N}},\left\Vert \cdot
\right\Vert )$ and all the $C_{p}(n)$ spaces, $1\leq p<\infty $, $n\in 
\mathbb{N}$) and all $2^{\aleph _{0}}$-dimensional normed spaces over $F$
(especially, the Banach spaces $l_{p}$ and $L_{p}(n)$, $1\leq p\leq \infty $%
, $n\in \mathbb{N}$, the subspaces $c_{0}\trianglelefteq c\trianglelefteq
l_{\infty }$ and Sobolev spaces $W_{p}^{(k)}(\Omega _{n})$) belong to the
same and unique non-trivial quotient shape type with respect to Banach
spaces, that is the finite one. A representing space may be the $\aleph _{0}$%
-dimensional unitary direct sum normed subspace $F_{0}^{\mathbb{N}}(2)$ of $%
l_{2}$ as well as the Hilbert space $l_{2}$.
\end{corollary}

\begin{proof}
Every infinite-dimensional separable Banach space has the (algebraic)
dimension $2^{\aleph _{0}}$. If an infinite-dimensional separable normed
space $X$ is not complete, then one may use the inclusion $%
j_{X}:X\rightarrow Cl(X)$, where $X$ is isometrically embedded into its
second dual space. Then $Cl(X)$ is an infinite-dimensional separable Banach
space, and hence, $\dim Cl(X)=2^{\aleph _{0}}$. All the considered concrete
spaces belong to the mentioned classes. Thus, the conclusion follows by
Lemma 1 (or Theorem 3 [15]) and Theorem 2.
\end{proof}

\begin{remark}
\label{R2}(i) A counterexample for the converse in Theorem 2 (i), in the
dimensional pair $\{\aleph _{0},2^{\aleph _{0}}\}$, is $X=(F_{0}^{\mathbb{N}%
},\left\Vert \cdot \right\Vert )$ and $Y=Cl(X)$ - the Banach completion of $%
X $ in the second dual space of $X$. Namely, by Lemmata 1 and 2 ($CH$
accepted) $Sh_{\text{\b{0}}}(X)=Sh_{\text{\b{0}}}(Y)$ and $\dim X=\aleph
_{0}<2^{\aleph _{0}}=\dim Y$.

\noindent (ii) Notice that, although each direct sum normed space $(F_{0}^{%
\mathbb{N}},\left\Vert \cdot \right\Vert )$ may represent the (unique)
finite quotient shape type of all $2^{\aleph _{0}}$-dimensional normed
spaces considered in Corollary 1, none of ($F_{0}^{\mathbb{N}},\left\Vert
\cdot \right\Vert )$ can represent any (but its own) of their countable
quotient shape types with respect to \emph{normed} spaces. Namely, ($F_{0}^{%
\mathbb{N}},\left\Vert \cdot \right\Vert )$ itself represents its own
countable quotient shape type, while the countable quotient shape types of
all normed (Banach) spaces $X$, $\dim X=2^{\aleph _{0}}$, are
non-rudimentary. (They reduce to their unique non-rudimentary finite
quotient shape type). Hence, \textquotedblleft
philosophically\textquotedblright\ speaking, in the \textquotedblleft world
of Banach spaces\textquotedblright\ there is no \textquotedblleft
fine/close\textquotedblright approximation of a $2^{\aleph _{0}}$-object by
the \textquotedblleft shape-like\textquotedblright\ $\aleph _{0}$-objects,
i.e., there is only a \textquotedblleft coarse\textquotedblright\
approximation by the finite-dimensional objects. It might be the main cause
for the general difficulties in a practical application, especially, in
solving of partial differential equations!? So the theoretical
\textquotedblleft advantage\textquotedblright\ (there is no $\aleph _{0}$%
-dimensional Banach space) can turn back to be a practical disadvantage.
\end{remark}

\textbf{Question 2. }\emph{Given a Banach space }$X$\emph{\ and its closed
subspace }$Z$\emph{\ such that }$\dim Z=\dim X\geq 2^{\aleph _{k}}$\emph{\
and }$\dim (X/Z)\leq \aleph _{k}$\emph{\ (}$k\geq 1$\emph{; }$GCH$\emph{\
accepted), is there a closed complement of }$Z$\emph{\ in }$X?$

\section{Application}

It is well known that, in general, a continuous linear function of a
(closed) subspace of a normed space into a Banach space of dimension $\dim
\geq 2$ does not admit a continuous linear extension on the whole space. We
shall prove that under certain dimensional conditions such an extension
exists. Clearly, if the dimension of a codomain space is $\dim =1$, then an
extension exists without any additional condition (the Hahn-Banach theorem).

\begin{theorem}
\label{T3}Let $X$ be a normed space having $\dim X\geq \kappa \geq \aleph
_{0}$, and let $Z\trianglelefteq X$ be a (normed) subspace. If, by the
inclusion $e:Cl(Z)\hookrightarrow X$, induced, quotient shape morphism $%
S_{\kappa ^{-}}(e):Cl(Z)\rightarrow X$ is an isomorphism of $S_{\kappa ^{-}}(%
\mathcal{N})$, then for every Banach space $Y$, over the same field, such
that $\dim Y<\varkappa $, every continuous linear function $f:Z\rightarrow Y$
admits a continuous linear extension $\bar{f}:X\rightarrow Y$. Moreover, $%
\left\Vert \bar{f}\right\Vert =\left\Vert f\right\Vert $.
\end{theorem}

\begin{proof}
Let $S_{\text{$\kappa ^{-}$}}(e):Cl(Z)\rightarrow X$ be an isomorphism of $%
Sh_{\text{$\kappa ^{-}$}}(\mathcal{N})$. Let $Y\in Ob(\mathcal{B}_{\text{$%
\kappa ^{-}$}})$ and let $f\in \mathcal{N}(Z,Y)$. We may assume, without
loss of generality, that $S_{\text{$\kappa ^{-}$}}(e)$ is represented by a
level morphism

$(1_{N},u_{\nu }):\boldsymbol{Z}_{\kappa ^{-}}^{\prime }=(Z_{\nu }^{\prime
},r_{\nu \nu ^{\prime }}^{\prime },N)\rightarrow (X_{\nu }^{\prime },p_{\nu
\nu ^{\prime }}^{\prime },N)=\boldsymbol{X}_{\kappa ^{-}}^{\prime }$

\noindent of $inv$-$\mathcal{N}_{\text{$\kappa ^{-}$}}$, where $\boldsymbol{Z%
}_{\kappa ^{-}}^{\prime }$ and $\boldsymbol{X}_{\kappa ^{-}}^{\prime }$ are
the expansion objects of $\kappa ^{-}$-expansions

$\boldsymbol{r}_{\kappa ^{-}}^{\prime }=(r_{\nu }^{\prime
}):Cl(Z)\rightarrow \boldsymbol{Z}_{\kappa ^{-}}^{\prime }$ \quad and

$\boldsymbol{p}_{\kappa ^{-}}^{\prime }=(p_{\nu }^{\prime }):X\rightarrow 
\boldsymbol{X}_{\kappa ^{-}}^{\prime }$

\noindent of $Cl(Z)$ and $X$ respectively. Then, for every $\nu \in N$, $%
p_{\nu }^{\prime }e=u_{\nu }r_{\nu }^{\prime }$ and, for every related pair $%
\nu \leq \nu ^{\prime }$ in $N$, $u_{\nu }r_{\nu \nu ^{\prime }}^{\prime
}=p_{\nu \nu ^{\prime }}^{\prime }u_{\varpi ^{\prime }}$. By the well-known
Morita lemma, for every $\nu \in N$, there exist a $\nu ^{\prime }\geq \nu $
and a $v_{\nu }\in \mathcal{N}(X_{\nu ^{\prime }}^{\prime },Z_{\nu }^{\prime
})$ such that the diagram

$%
\begin{array}{ccc}
Z_{\nu }^{\prime } & \overset{r_{\nu \nu ^{\prime }}^{\prime }}{\leftarrow }
& Z_{\nu ^{\prime }}^{\prime } \\ 
u_{\nu }\downarrow & \nwarrow v_{\nu } & \downarrow u_{\nu ^{\prime }} \\ 
X_{\nu }^{\prime } & \overset{p_{\nu \nu ^{\prime }}^{\prime }}{\leftarrow }
& X_{\nu ^{\prime }}^{\prime }%
\end{array}%
$

\noindent in $\mathcal{N}$ commutes, i.e., $u_{\nu }v_{\nu }=p_{\nu \nu
^{\prime }}^{\prime }$ and $v_{\nu }u_{\nu ^{\prime }}=r_{\nu \nu ^{\prime
}}^{\prime }$. Denote by $i:Z\hookrightarrow Cl(Z)$ the inclusion, and by $%
f^{\prime }:Cl(Z)\rightarrow Y$ the unique continuous linear extension of $f$%
, i.e., $f^{\prime }i=f$ (Lemma 1). Since $\boldsymbol{r}^{\prime }$ is $%
\kappa ^{-}$-expansion of $Cl(Z)$ and $\dim Y<\kappa $, there exist a $\nu
\in N$ and an $f^{\nu }\in \mathcal{N}(Z_{\nu }^{\prime },Y)$ such that $%
f^{\nu }r_{\nu }^{\prime }=f^{\prime }$. Put

$\bar{f}=f^{\nu }v_{\nu }p_{\nu ^{\prime }}^{\prime }:X\rightarrow Y$.

\noindent Then

$\bar{f}e=f^{\nu }v_{\nu }p_{\nu ^{\prime }}^{\prime }e=f^{\nu }v_{\nu
}u_{\nu ^{\prime }}r_{\nu ^{\prime }}^{\prime }=f^{\nu }r_{\nu \nu ^{\prime
}}^{\prime }r_{\nu ^{\prime }}^{\prime }=f^{\nu }r_{\nu }^{\prime
}=f^{\prime }$.

\noindent Consequently, $\bar{f}(ei)=(\bar{f}e)i=f^{\prime }i=f$, that is a
desired extension. Recall that all the projections and bonding morphisms in
a canonical expansion are the appropriate non-trivial quotient morphisms
(the initial one onto $\{\theta \}$ may be dropped and ignored), and hence,
their norm is $1$. Further, since $e:Cl(Z)\hookrightarrow X$ is the
inclusion, the induced canonical $(\varphi ,e_{\lambda })$ inv-morphism
consists of the quotient morphisms as well, and thus, $\left\Vert e_{\lambda
}\right\Vert =1$, $\lambda \in \Lambda _{\kappa ^{-}}$ (those onto the
initial term or from the initial term may be dropped and ignored). Further,
the construction of a level morphism uses the already existing morphisms
only. Therefore,

$\left\Vert u_{\nu }\right\Vert =\left\Vert u_{\nu ^{\prime }}\right\Vert
=\left\Vert v_{\nu }\right\Vert =1,\nu ,\nu ^{\prime }\in N$.

\noindent Finally, $\left\Vert e\right\Vert =1$, and the extension $%
f^{\prime }$ of $Z$ on $Cl(Z)$ does not affect the norm $\left\Vert
f\right\Vert $. Consequently,

$\left\Vert \bar{f}\right\Vert =\left\Vert f^{\nu }v_{\nu }p_{\nu ^{\prime
}}^{\prime }\right\Vert =\left\Vert f^{\nu }\right\Vert =\left\Vert
f\right\Vert $,

\noindent that completes the proof.
\end{proof}

\begin{theorem}
\label{T4}Let $X$ be a normed space having $\dim X=\kappa \geq \aleph _{0}$,
and let $Z\trianglelefteq X$ be a (normed) subspace such that $\dim
Cl(Z)=\dim X$ and $\dim (X/Cl(Z))<\kappa $, and let $Y$ be a Banach space
(over the same field) having $\dim Y<\kappa $.

\noindent (i) If $X$ is a separable Banach space and $\max \{\dim
(X/Cl(Z),\dim Y)\leq \aleph _{0}$, then every continuous linear function $%
f:Z\rightarrow Y$ admits a continuous linear extension $\bar{f}:X\rightarrow
Y$ such that $\left\Vert \bar{f}\right\Vert =\left\Vert f\right\Vert $.

\noindent (ii) If $X$ is a bidual-like Banach space, then every continuous
linear function $f:Z\rightarrow Y$ admits a continuous linear extension $%
\bar{f}:X\rightarrow Y$ such that $\left\Vert \bar{f}\right\Vert =\left\Vert
f\right\Vert $.
\end{theorem}

\begin{proof}
Firstly notice that, in general, a desired (unique) ontinuous linera
extension on $Cl(Z)$ exists by Lemma 1. Further, in the case of $\dim
(X/Cl(Z))<\aleph _{0}$ there is no need for the assumption $\dim Cl(Z)=\dim
X $. Namely, $\dim X\geq \aleph _{0}$ and $\dim (X/Cl(Z))<\aleph _{0}$ imply 
$\dim Cl(Z)=\dim X$. Furthermore, there is a rather simple proof of that
special case without using Theorem 3. Nevertheless, we want to use Theorem 3
in our proof. Denote by $e:Cl(Z)\hookrightarrow X$ the (closed continuous
linear) inclusion. By Theorem 2 (i), it follows that $Sh_{\text{\b{0}}%
}(Cl(Z))=Sh_{\text{\b{0}}}(X)$.

\noindent (i). Since $X\in Ob(s\mathcal{B})$, it follows that either $\dim
X<\aleph _{0}$ or $\dim X=2^{\aleph _{0}}$. In the first
(finite-dimensional) case, the statement is obviously true. Let $\dim
X=2^{\aleph _{0}}$. Since $Y$ is a Banach space having $\dim Y<\dim X$, it
follows that $\dim Y<\aleph _{0}$. According to Theorem 3, it suffices to
prove that the induced quotient shape morphism

$F\equiv S_{\text{\b{0}}}(e):Cl(Z)\rightarrow X$

\noindent of $Sh_{\text{\b{0}}}(\mathcal{B})\subseteq Sh_{\text{\b{0}}}(%
\mathcal{N})$ is an isomorphism. We shall prove this by proving the analogue
claim in the proof of statement (ii).

\noindent (ii). Since $X\in Ob(b\mathcal{B})$ and $e:Cl(Z)\rightarrow X$ a
closed embedding such that $\dim Cl(Z)=\dim X$ and $\dim (X/Cl(Z))<\dim
X=\kappa $, Theorem 2 (iii) implies that $Sh_{\kappa ^{-}}(Cl(Z))=Sh_{\kappa
^{-}}(X)$. According to Theorem 3, it remains to prove that the quotiemt
shape morphism

$F\equiv S_{\kappa ^{-}}(e):Cl(Z)\rightarrow X$

\noindent of $S_{\kappa ^{-}}(b\mathcal{B})\subseteq S_{\kappa ^{-}}(%
\mathcal{N})$ is an isomorphism. In order to cover the both statements,
assume that $X$ is a Banach space. Let

$\boldsymbol{r}_{\text{$\kappa ^{-}$}}=(r\mu ):Cl(Z)\rightarrow \boldsymbol{Z%
}_{\text{$\kappa ^{-}$}}=(Z_{\mu },r_{\mu \mu ^{\prime }},M_{\text{$\kappa
^{-}$}})$ \quad and

$\boldsymbol{p}_{\text{$\kappa ^{-}$}}=(p_{\lambda }):X\rightarrow 
\boldsymbol{X}_{\text{$\kappa ^{-}$}}=(X_{\lambda },p_{\lambda \lambda
^{\prime }},\Lambda _{\text{$\kappa ^{-}$}})$

\noindent be the canonical $\mathcal{N}_{\text{$\kappa ^{-}$}}$-expansions
of $Cl(Z)$ and $X$ respectively. (They are, actually, the $\mathcal{B}%
_{\kappa ^{-}}$-expansions, because $Cl(Z)$ is a Banach space as well, and,
consequently, all the quotient spaces by closed subspaces are Banach
spaces.) Then the induced pro-morphism $\boldsymbol{e}_{\kappa ^{-}}:%
\boldsymbol{Z}_{\kappa ^{-}}\rightarrow \boldsymbol{X}_{\kappa ^{-}}$ of $e$
is the equivalence class of the inv-morphism

$(\varphi ,e_{\lambda }):\boldsymbol{Z}_{\text{$\kappa ^{-}$}}\rightarrow 
\boldsymbol{X}_{\text{$\kappa ^{-}$}}$,

\noindent where $\varphi :\Lambda _{\text{$\kappa ^{-}$}}\rightarrow M_{%
\text{$\kappa ^{-}$}}$ is obtained by the expansion factorization property
(E1) for each $p_{\lambda }e=e_{\lambda }r_{\varphi (\lambda )}$. Therefore,
by the construction of a canonical $\mathcal{N}_{\text{$\kappa ^{-}$}}$%
-expansion, for every $\lambda \in \Lambda _{\text{$\kappa ^{-}$}}$,
corresponding to a closed subspace $U_{\lambda }\trianglelefteq X$ such that 
$\dim U_{\lambda }=\dim X$ and $\dim (X/U_{\lambda })<\kappa $, the class $%
[z]_{\varphi (\lambda )}=z+(Cl(Z)\cap U_{\lambda })\in Z_{\varphi (\lambda
)} $ goes by $e_{\lambda }:Z_{\varphi (\lambda )}\rightarrow X_{\lambda }$
to the class $[z]_{\lambda }=z+U_{\lambda }\in X_{\lambda }$. Further, if $%
\lambda \leq \lambda ^{\prime }$, then $U_{\lambda ^{\prime
}}\trianglelefteq U_{\lambda }$ (having the same dimensional properties),
and hence, $Cl(Z)\cap U_{\lambda ^{\prime }}\trianglelefteq Cl(Z)\cap
U_{\lambda }$ implying $\varphi (\lambda )\leq \varphi (\lambda ^{\prime })$
and $e_{\lambda }r_{\varphi (\lambda )\varphi (\lambda ^{\prime
})}=p_{\lambda \lambda ^{\prime }\lambda }e_{\varphi (\lambda ^{\prime })}$.
On the other side, every $\mu \in M_{\text{$\kappa ^{-}$}}$, corresponding
to a closed subspace $V_{\mu }\trianglelefteq Cl(Z)$ such that $\dim V_{\mu
}=\dim Cl(Z)$ and $\dim (Cl(Z)/V_{\mu })<\kappa $, is a $\lambda _{\mu }\in
\Lambda _{\text{$\kappa ^{-}$}}$, corresponding to a closed $U_{\lambda
_{\mu }}=V_{\mu }\trianglelefteq X$ such that $\dim U_{\lambda _{\mu }}=\dim
X$ and $\dim (X/U_{\lambda _{\mu }})<\kappa $. Thus, there exists a
canonical \textquotedblleft inclusion\textquotedblright\ $\psi :M_{\text{$%
\kappa ^{-}$}}\hookrightarrow \Lambda _{\text{$\kappa ^{-}$}}$, $\psi (\mu
)=\lambda _{\mu }$. Put $\psi \lbrack M_{\text{$\kappa ^{-}$}}]\equiv
\Lambda _{\text{$\kappa ^{-}$}}^{\prime }\subseteq \Lambda _{\text{$\kappa
^{-}$}}$. It remains to prove that $\psi $ is a cofinal function. Indeed, in
that case, the codomain restriction

$\boldsymbol{e}_{\text{$\kappa ^{-}$}}^{\prime }:\boldsymbol{Z}_{\text{$%
\kappa ^{-}$}}\rightarrow \boldsymbol{X}_{\text{$\kappa ^{-}$}}^{\prime }$

\noindent of $\boldsymbol{e}_{\text{$\kappa ^{-}$}}$, i.e., the equivalence
class of the codomain restriction

$(\varphi ^{\prime },e_{\lambda }):\boldsymbol{Z}_{\text{$\kappa ^{-}$}%
}\rightarrow \boldsymbol{X}_{\text{$\kappa ^{-}$}}^{\prime }$,

\noindent of $(\varphi ,e_{\lambda })$, where the bijection $\varphi
^{\prime }:\Lambda _{\text{$\kappa ^{-}$}}^{\prime }\rightarrow M_{\text{$%
\kappa ^{-}$}}$, $\varphi ^{\prime }(\lambda =\lambda _{\mu })=\mu $, is the
domain restriction of the index function $\varphi $, will be an isomorphism
of $pro$-$\mathcal{B}$, and consequently,

$S_{\text{$\kappa ^{-}$}}(e)=\left\langle \boldsymbol{e}_{\text{$\kappa ^{-}$%
}}\right\rangle =\left\langle \boldsymbol{e}_{\text{$\kappa ^{-}$}}^{\prime
}\right\rangle :Cl(Z)\rightarrow X$

\noindent will be a desired quotient shape isomorphism. Namely, in that
case, the restriction

$\boldsymbol{p}_{\text{$\kappa ^{-}$}}^{\prime }=(p_{\lambda }):X\rightarrow 
\boldsymbol{X}_{\text{$\kappa ^{-}$}}^{\prime }=(X_{\lambda },p_{\lambda
\lambda ^{\prime }},\Lambda _{\text{$\kappa ^{-}$}}^{\prime })$

\noindent will be an $\mathcal{N}_{\text{$\kappa ^{-}$}}$-expansion of $X$.
The proof now reduces to the following claim:

\emph{For every closed subspace }$U\trianglelefteq X$\emph{\ such that }$%
\dim U=\dim X$\emph{\ and }$\dim (X/U)<\kappa $\emph{, there exists a closed
subspace }$V\trianglelefteq Cl(Z)$\emph{\ such that }$\dim V=\dim Cl(Z)$%
\emph{\ and }$\dim (Cl(Z)/V)<\kappa $\emph{, and in addition, }$%
V\trianglelefteq U$\emph{, }$\dim V=\dim X$\emph{\ and }$\dim (X/V)<\kappa $%
\emph{. }

\noindent Given such a $U$, put $V=Cl(Z)\cap U$, and the verification works
in the same way as in the proof of Lemma 3 (ii).
\end{proof}

At the end, we give a rather general example as a corollary.

\begin{corollary}
\label{C2}(i) Given a $p,1\leq p<\infty $, let $X$ be a separable Banach
space that admits a continuous injection $i:l_{p}\rightarrow X$ such that $%
\dim (X/Cl(i[l_{p}]))\leq \aleph _{0}$, and let $Y$ be a finite-dimensional
normed space over the same field. Then, for each $r,$ $1\leq r\leq p$, every
continuous linear function $f:F_{0}^{\mathbb{N}}(r)\rightarrow Y$ ($%
f:l_{r}\rightarrow Y$ as well), admits a continuous linear extension $\bar{f}%
:X\rightarrow Y$ through all the $l_{r}$ ($l_{s},$ $r\leq s\leq p$);

\noindent (ii) Let $X$ be a separable Banach space that admits a continuus
injection $i:c_{0}\rightarrow X$ such that $\dim (X/Cl(i[c_{0}]))\leq \aleph
_{0}$, and let $Y$ be a finite-dimensional normed space over the same field.
Then every continuous linear function $f:F_{0}^{\mathbb{N}}(\infty
)\rightarrow Y$ ($f:l_{p}(\infty )\rightarrow Y$ as well, $1\leq p<\infty $%
,) admits a continuous linear extension $\bar{f}:X\rightarrow Y$ through all
the $l_{p}(\infty )$, $1\leq p<\infty $ ($l_{p^{\prime }}(\infty ),$ $p\leq
p^{\prime }<\infty $);
\end{corollary}

\begin{proof}
(i).Observe that the inclusions $i_{p}^{0r}:F_{0}^{\mathbb{N}}(r)\rightarrow
l_{p}$ and $i_{p}^{r}:l_{r}\rightarrow l_{p}$, $1\leq r\leq p<\infty $, are
continuous, and that, in $l_{p}$, $Cl(F_{0}^{\mathbb{N}}(r))=Cl(l_{r})=l_{p}$%
. Now apply Theorem 4 to $Z=F_{0}^{\mathbb{N}}(r)$ ($Z=l_{r}$).

\noindent (ii). Notice that

$F_{0}^{\mathbb{N}}(\infty )\trianglelefteq l_{1}(\infty )\trianglelefteq
\cdots \trianglelefteq l_{p}(\infty )\trianglelefteq \cdots \trianglelefteq
l_{p^{\prime }}(\infty )\trianglelefteq \cdots \trianglelefteq c_{0}$

\noindent are proper subspaces of $c_{0}$ and that $Cl(F_{0}^{\mathbb{N}%
}(\infty ))=l_{p}(\infty )$ in $l_{p}(\infty )$, and $Cl(F_{0}^{\mathbb{N}%
}(\infty ))=Cl(l_{p}(\infty ))=c_{0}$ in $l_{\infty }$, $1\leq p<\infty $.
By applying Theorem 4 to $Z=F_{0}^{\mathbb{N}}(\infty )$ and $Z=l_{p}(\infty
)$ respectively ($Y$ is a Banach space), the conclusion follows.

\noindent (Observe that, in $l_{\infty }$, $Cl(F_{0}^{\mathbb{N}%
}(p))=Cl(l_{p})=c_{0}$, and that $\dim (c/c_{0})=1$, where $c\trianglelefteq
l_{\infty }$ is the subspace of all convergent sequences in $F$, which is
closed. Hence, $X=c$ is a concrete example for (i) and (ii).) \bigskip
\end{proof}

\begin{center}
\textbf{References\smallskip }
\end{center}

\noindent \lbrack 1] K. Borsuk, \textit{Concerning homotopy properties of
compacta}, Fund. Math. \textbf{62} (1968), 223-254.

\noindent \lbrack 2] K. Borsuk, \textit{Theory of Shape}, Monografie
Matematyczne \textbf{59}, Polish Scientific Publishers, Warszawa, 1975.

\noindent \lbrack 3] J.-M. Cordier and T. Porter, \textit{Shape Theory:
Categorical Methods of Approximation}, Ellis Horwood Ltd., Chichester, 1989.
(Dover edition, 2008.)

\noindent \lbrack 4] J. Dugundji, \textit{Topology}, Allyn and Bacon, Inc.
Boston, 1978.

\noindent \lbrack 5] Dydak and J. Segal, \textit{Shape theory: An
introduction}, Lecture Notes in Math. \textbf{688}, Springer-Verlag, Berlin,
1978.

\noindent \lbrack 6] H. Herlich and G. E. Strecker, \textit{Category Theory,
An Introduction}, Allyn and Bacon Inc., Boston, 1973.

\noindent \lbrack 7] N. Kocei\'{c} Bilan and N. Ugle\v{s}i\'{c}, \textit{The
coarse shape}, Glasnik. Mat. \textbf{42}(\textbf{62}) (2007), 145-187.

\noindent \lbrack 8] E. Kreyszig, \textit{Introductory Functional Analysis
with Applications}, John Wiley \& Sons, New York, 1989.

\noindent \lbrack 9] S. Kurepa, \textit{Funkcionalna analiza : elementi
teorije operatora}, \v{S}kolska knjiga, Zagreb, 1990.

\noindent \lbrack 10] S. Marde\v{s}i\'{c} and J. Segal, \textit{Shape Theory}%
, North-Holland, Amsterdam, 1982.

\noindent \lbrack 11] W. Rudin, \textit{Functional Analysis, Second Edition}%
, McGraw-Hill, Inc., New York, 1991.

\noindent \lbrack 12] N. Ugle\v{s}i\'{c}, \textit{The shapes in a concrete
category}, Glasnik. Mat. Ser. III \textbf{51}(\textbf{71}) (2016), 255-306.

\noindent \lbrack 13] N. Ugle\v{s}i\'{c}, \textit{On the quotient shape of
vectorial spaces}, Rad HAZU - Matemati\v{c}ke znanosti, Vol. \textbf{21} = 
\textbf{532} (2017), 179-203.

\noindent \lbrack 14] N. Ugle\v{s}i\'{c}, \textit{On the quotient shapes of
topological spaces}, Top. Appl. \textbf{239} (2018), 142-151.

\noindent \lbrack 15] N. Ugle\v{s}i\'{c}, \textit{The quotient shapes of }$%
l_{p}$\textit{\ and }$L_{p}$\textit{\ spaces}, Rad HAZU. Matemati\v{c}ke
znanosti, Vol. \textbf{27} = \textbf{536} (2018), 149-174.

\noindent \lbrack 16] N. Ugle\v{s}i\'{c} and B. \v{C}ervar, \textit{The
concept of a weak shape type}, International J. of Pure and Applied Math. 
\textbf{39} (2007), 363-428.

\end{document}